\documentclass[a4paper]{amsart}
\usepackage{amsmath,amscd,amssymb,amsxtra}
\usepackage[all]{xy}
\usepackage{color}
\newtheorem{Theorem}{Theorem}
\newtheorem{Proposition}[Theorem]{Proposition}
\newtheorem{Lemma}[Theorem]{Lemma}
\newtheorem{Corollary}[Theorem]{Corollary}

\theoremstyle{definition}
\newtheorem{Definition}[Theorem]{Definition}
 \newtheorem{Example}[Theorem]{Example}

\numberwithin{equation}{section}

\newcommand{\Ss}{\Bbb{S}}
\newcommand{\N}{\Bbb{N}}
\newcommand{\C}{\bf C}
\DeclareMathOperator{\NGc}{\bf NG}
\DeclareMathOperator{\NGcc}{{\bf NGC}}
\newcommand{\holim}{{\underleftarrow{{\rm holim}}}}
\newcommand{\cholim}{{\underrightarrow{{\rm lim}}}}
\newcommand{\inlim}{{\underrightarrow{{\rm lim}}}}
\newcommand{\pnlim}{{\underleftarrow{{\rm lim}}}}

\DeclareMathOperator{\CF}{\check{C}}
\DeclareMathOperator{\FF}{\check{F}}

\DeclareMathOperator{\Top}{Top}

\DeclareMathOperator{\map}{map}

\begin{document}

\title[Steenrod-\v{C}ech homology-cohomology theories]{Steenrod-\v{C}ech homology-cohomology theories associated with bivariant functors}%

\author{Kohei Yoshida}%
\email{sc09121@s.okadai.jp}
\date{\today}%
\begin{abstract}
Let $\NGc_0$ denote the category of all pointed 
numerically generated spaces and
 continuous maps preserving base-points. 
In \cite{numerical}, we described a passage from bivariant functors to
generalized homology and cohomology theories. 
In this paper, we construct a bivariant functor 
such that the associated cohomology is the $\CF$ech
cohomology and the homology is the Steenrod homology
(at least for compact metric spaces). 
\end{abstract}
\maketitle
%
\section{introduction}
We call a topological space $X$ numerically generated
if it has the final topology with respect to its singular simplexes.
CW-complexes are typical examples of such numerically generated spaces.
Let  $\NGc_0$ be the
category of pointed numerically
generated spaces and 
pointed continuous maps.
In \cite{numerical} we showed that $\NGc_0$
is a symmetric monoidal closed category with respect to the smash
product,
and that every bilinear enriched functor
$F:\NGc_0^{op}\times \NGc_0 \to \NGc_0$
gives rise to a pair of generalized homology and cohomology theories,
denoted by $h_{\bullet}(-,F)$ and $h^{\bullet}(-,F)$ respectively,
such that
\[
h_n(X,F) \cong \pi_0 F(S^{n+k},\Sigma^k X), \quad
h^n(X,F) \cong \pi_0 F(\Sigma^k X,S^{n+k})
\]
hold whenever $k$ and $n+k$ are non-negative.

As an example, consider the bilinear enriched functor $F$
which assigns to $(X,Y)$ 
the mapping space from $X$ to the topological free abelian group $AG(Y)$
generated by the points of $Y$ modulo the relation $*\sim 0$.
The Dold-Thom theorem says that if $X$ is a CW-complex then 
the groups $h_n(X,F)$ and $h^n(X,F)$ are, respectively, isomorphic to 
the singular homology and cohomology groups of $X$.
But this is not the case for general $X$;
there exists a space $X$ such that $h_n(X,F)$ (resp.\ $h^n(X,F)$)
is not isomorphic to the singular homology (resp.\ cohomology) group of $X$.

The aim of this paper is to construct a bilinear enriched functor
such that for any space $X$ the associated cohomology groups are
isomorphic to the $\CF$ech cohomology groups of $X$.
Interestingly, it turns out that the corresponding homology groups
are isomorphic to the Steenrod homology groups for
any compact metrizable space $X$.
Thus we obtain a bibariant theory which ties together
the $\CF$ech cohomology and the Steenrod homology theories.

Recall that the $\CF$ech cohomology group of $X$
with coefficients group $G$ is defined to be the colimit
of the singlular cohomology groups
\[
\textit{\v{H}}^n(X,G) = \inlim_\lambda H^n(X^{\CF}_\lambda,G),
\]
where $\lambda$ runs through coverings of $X$ and
$X^{\CF}_\lambda$ is the \v{C}ech nerve corresponding to $\lambda$., i.e.
$v\in X_\lambda^{\CF}$ is a
vertex of $X^{\CF}_\lambda$ corresponding to an open set $V\in \lambda$.
On the other hand, the Steenrod homology group of a compact metric space
$X$ is defined as follows.
As $X$ is a compact metric space, there is a sequence 
$\{\lambda_i \}_{i\geq 0}$ of finite open covers of $X$ such that
$\lambda_0 = \{X\}$, $\lambda_i$ is a refinement of $\lambda_{i-1}$, and
$X$ is the inverse limit $\pnlim_i X^{\CF}_{\lambda_i}$.
According to \cite{F}, the Steenrod homology group of $X$ with
coefficients in the spectrum $ \Ss $ is defined to be the group
\[
H^{st}_{n}(X,\Ss) =
\pi_{n}\underleftarrow{{\rm holim}}_{\lambda_i}(X^{\CF}_{\lambda_i}\wedge \Ss )
\]
where $\underleftarrow{{\rm holim}}$ denotes the homotopy inverse limit.

Let $\NGcc_0$ be the subcategory of pointed numerically
generated compact metric spaces and pointed continuous maps.
For given a linear enriched functor $T:\NGc_0\to\NGc_0$,
let
\[
\mbox{\v{F}}:\NGc^{op}_0 \times \NGcc_0 \to \NGc_0
\]
be a bifunctor which maps $(X,Y)$ to the space $\inlim_\lambda \map_0(X_\lambda,\holim_{\mu_i} T(Y_{\mu_i}^{\CF})).$
Here $\lambda$ runs through coverings of $X$, and
$X_\lambda$ is the Vietoris nerves corresponding to $\lambda$ (\cite{P}).
The main results of the paper can be stated as follows.
\begin{Theorem}\label{TH2}
The functor $\FF$ is a bilinear enriched functor.
\end{Theorem}

\begin{Theorem}\label{Sp}
Let $X$ be a compact metraizable space.
Then $h_{n}(X,\FF )=H_n(X,\Ss )$ is
 the Steenrod homology group with coefficients in the spectrum $\Ss =\{ T(S^k) \}$.
\end{Theorem}

In particular, let us take $AG$ as $T$,
and denote
\[
\CF:\NGc_0^{op}\times \NGcc_0 \to \NGc_0
\]
be the corresponding bifunctor \v{F}.
\begin{Theorem}\label{Cor1}
For any pointed space $X$,
$h^n(X,\CF)$ is the $\CF$ech cohomology
group of $X$, and $h_n(X,\CF)$ is the Steenrod homology group
of $X$ if $X$ is a compact metralizable space. 
\end{Theorem}

Recall that the Steenrod homology group
 is related to the $\CF$ech homology group of $X$ by the exact sequence

\[
\xymatrix{
0 \ar[r] &\pnlim^1_{\lambda_i} \tilde{H}_{n+1}(X^{\CF}_{\lambda_i})  \ar[r] &
 H^{st}_n(X) \ar[r] 
& \tilde{H}_{n}(X) \ar[r] & 0 .\\ 
}
\]
If $X$ is a movable compactum then we have
$\pnlim^1_{\lambda_i} \tilde{H}_{n+1}(X^{\CF}_{\lambda_i})=0$, and hence the following corollary follows.

\begin{Corollary}\label{Cor}
Let $X$ be a movable compactum.
Then $h_n(X,\CF )$ is the $\CF$ech homology group of $X$.
\end{Corollary}

The paper is organized as follows. In Section 2 we 
recall from \cite{numerical} the category $\NGc_0$ and the passage
from bilinear enriched functors to generalized homology
and cohomology theories. And we recall that Vietoris and 
$\CF$ech nerves; 
In Section 3 we prove Theorem \ref{TH2}; 
Finally, in Section 4 we prove Theorems 2 and \ref{Cor1}.

\bigskip

\paragraph{Acknowledgement}
We thank Professor K.~Shimakawa for calling my attention to the subject
and for 
useful conversations while preparing the manuscript.

\section{Preliminaries}
\subsection{Homology and cohomology theories via bifunctors}

Let $\NGc_0$ be the category of pointed numerically
generated topological spaces and 
pointed continuous maps.
In \cite{numerical} we showed that $\NGc_0$ satisfies the following properties:
\begin{enumerate}
\item It contains pointed CW-complexes;
\item It is complete and cocomplete;
\item It is monoidally closed in the sense that there is an internal hom
  $Z^Y$ satisfying a natural bijection
  $\hom_{\NGc_0}(X \wedge Y,Z) \cong \hom_{\NGc_0}(X,Z^Y)$;
\item There is a coreflector $\nu \colon \Top_0 \to \NGc_0$ such that
  the coreflection arrow $\nu{X} \to X$ is a weak equivalence;
\item The internal hom $Z^Y$ is weakly equivalent to the space
  of pointed maps from $Y$ to $Z$ equipped with the compact-open topology.
\end{enumerate}
Throughout the paper, we write $\map_0(Y,Z)=Z^Y $ for any
$Y,\ Z \in \NGc_0$.

A map $f \colon X \to Y$ between topological spaces is said to be
numerically continuous if the composite
$f\circ \sigma \colon \Delta^n \to Y$ is continuous
for every singular simplex $\sigma \colon \Delta^n \to X$.
We have the following.

\begin{Proposition}{$($\cite{numerical}$)$}\label{5}
  Let $f \colon X \to Y$ be a map between numerically generated spaces.
  Then  $f$ is numerically continuous
  if and only if $f$ is continuous.
\end{Proposition}

From now on, we assume that $\C_0$ satisfies the following conditions:
(i) $\C_0$ contains all finite CW-complexes.
(ii) $\C_0$ is closed under finite wedge sum. 
(iii) If $A \subset X$ is an inclusion of objects in $\C_0$ then
its cofiber $X \cup CA$ belongs to $\C_0$;
in particular, $\C_0$ is closed under the suspension functor
$X \mapsto \Sigma X$.

\begin{Definition}
Let $\C_0$ be a full subcategory of $\NGc_0$.  A functor $T \colon \C_0 \to \NGc_0$ is called
  {\em enriched $($or continuous$)$ \/} if the map 
\[
T:\map_0(X,X')\to \map_0(T(X),T(X')),
\] 
which assigns $T(f)$ to every $f$,
is a pointed continuous map.
\end{Definition}
Note that if $f$ is constant,
then so is $T(f)$.

\begin{Definition}
An enriched functor $T$ is called $linear$ if 
for any pair of a pointed space $X$, a
 sequence 
\[
T(A)\to T(X)\to T(X\cup CA)
\]
induced by the cofibration sequence 
$A\to X \to X\cup CA$, is a homotopy fibration sequence.
\end{Definition}
\begin{Example}
Let $AG:\mbox{CW}_0\to \NGc_0$ be the functor which assigns to a pointed 
CW-complex
$(X,x_0)$ the topological abelian group $AG(X)$ generated by
the points of $X$ modulo the relation $x_0 \sim 0$.
Then $AG$ is a linear enriched functor. (see \cite{numerical})
\end{Example}

\begin{Theorem}{$($\cite[Th 6.4]{numerical}$)$}
A linear enriched functor
$T$ defines a generalized homology $\{ h_n(X,T) \}$ satisfying 
\begin{equation*}
h_n(X,T)= \begin{cases}
            \pi_nT(X), & n\ge 0\\
            \pi_0T(\Sigma^{-n}X), & n<0.
            \end{cases}\\
\end{equation*}
\end{Theorem}

Next we introduce the notion of a bilinear enriched functor, and describe a passage
from a bilinear enriched functor to generalized cohomology and generalized homology theories.
We assume that $\C_0'$ satisfies the same conditions of 
$\C_0$.
\begin{Definition}\label{dual}
Let $\C_0$ and $\C_0'$ be full subcategories of $\NGc_0$.
A {\em bifunctor} $F \colon \C_0^{op}\times \C_0' \to
 \NGc_0$ is a function which
\begin{enumerate}
\item to each objects $X\in \C_0$ and $Y\in \C_0'$
 assigns an object $F(X,Y)\in \NGc_0$;
\item to each $f\in \map_0(X,X')$, $g\in \map_0(Y,Y')$ assigns a continuous map $F(f,g)\in \map_0(F(X',Y),~F(X,Y'))$.

$F$ is required to satisfy the following equalities: 
\begin{enumerate}
\item $F(1_X,1_Y)=1_{F(X,Y)}$;
\item $F(f,g)=F(1_X,g)\circ F(f,1_Y)=F(f,1_{Y'})\circ F(1_{X'},g)$;
\item $F(f'\circ f,1_Y)=F(f,1_Y)\circ F(f',1_Y)$,
  $F(1_X,g'\circ g)=F(1_X,g')\circ F(1_X,g)$.
\end{enumerate}
\end{enumerate}
\end{Definition}

\begin{Definition}
  A bifunctor $F \colon \C_{0}^{op}\times \C_0 \to \NGc_0$ is called
  {\em enriched\/} if the map 
\[
F:\map_0(X,X')\times \map_0(Y,Y')\to \map_0(F(X',Y),F(X,Y')),
\] 
which assigns $F(f,g)$ to every pair $(f,g)$,
is a pointed continuous map.
\end{Definition}
Note that if either $f$ or $g$ is constant,
then so is $F(f,g)$.
\begin{Definition}\label{biexact}
For any pairs of pointed spaces $(X,A)$ and $(Y,B)$, $F$ is $bilinear$ if the
 sequences 
\begin{enumerate}
\item $F(X\cup CA,Y)\to F(X,Y)\to F(A,Y)$
\item $F(X,B)\to F(X,Y)\to F(X,Y\cup CB)$,
\end{enumerate}
induced by the cofibration sequences 
$A\to X \to X\cup CA$ and
$B\to Y \to Y\cup CB$, are homotopy fibration
sequences.
\end{Definition}
\begin{Example}
Let $T:\NGc_0\to \NGc_0$ be a linear enriched functor,
and let $F(X,Y)=\map_0(X,T(Y))$ for $X,Y\in \NGc_0$.
Then $F:\NGc_0^{op}\times \NGc_0 \to \NGc_0$ is a bilinear enriched functor.
\end{Example}

\begin{Theorem}{$($\cite[Th 7.4]{numerical}$)$}
A bilinear enriched functor
$F$ defines a generalized cohomology $\{ h^n(-,F) \}$ and
a generalized homology $\{ h_n(-,F) \}$ such that 
\begin{equation*}
h_n(Y,F)= \begin{cases}
            \pi_0F(S^n,Y) & n\ge 0\\
            \pi_0F(S^0,\Sigma^{-n}Y) & n<0,
            \end{cases}\\
~~~~h^n(X,F)= \begin{cases}
            \pi_0F(X,S^n) & n\ge 0\\
            \pi_{-n}F(X,S^0) & n<0,
            \end{cases}
\end{equation*}
hold for any $X\in \C_0$ and $Y\in \C_0'$.
\end{Theorem}
\begin{Proposition}\label{numerical}{$($\cite{numerical}$)$}
If $X$ is a $CW$-complex, we have 
$h_n(X,F) =H_n(X,\Ss )$ and $h^n(X,F)=H^n(X,\Ss)$,
the generalized homology and cohomology groups
with coefficients in the spectrum $\Ss =\{ F(S^0,S^n)~|~n\geq 0 \}$.
\end{Proposition}

\subsection{Vietoris and $\CF$ech nerves}
For each $X \in \NGc_{0}$,~let $\lambda$ be an open covering of $X$.
According to \cite{P}, the Vietoris nerve of $\lambda$ 
is a simplicial set in which an
$n$-simplex is an ordered $(n+1)$-tuple
$(x_0,x_1,\cdots ,x_n)$ of points contained in 
an open set $U\in \lambda$.
Face and degeneracy operators are respectively given by 
\[
d_i(x_0,\cdots ,x_n)=(x_0,x_1,\cdots ,x_{i-1},x_{i+1},
\cdots ,x_n)
 \] 
 and
 \[
 s_i(x_0,x_1,\cdots x_n)=(x_0,x_1,\cdots x_{i-1},x_i,
 x_i,x_{i+1},\cdots ,x_n),~~0\le i \le n.
 \]
We denote the realization of
 the Vietoris nerve of $\lambda$ by $X_{\lambda}$. 
If $\lambda$ is a refinement of $\mu $,
then there is a canonical map $\pi_{{\mu}}^{{\lambda}}:X_{\lambda}\to X_{\mu}$ 
induced by the identity map of $X$.

 The relation between the Vietoris and the $\CF$ech nerves
 is given by the following Proposition due to Dowker.

\begin{Proposition}\label{D}$ ($\cite{D}$) $
The $\CF$ech nerve $X_\lambda^{\CF}$ and the Vietoris nerve $X_\lambda$ have the same homotopy type.
\end{Proposition}

According to \cite{D}, for arbitrary topological space, 
the Vietoris and $\CF$ech homology groups are isomorphic 
and the Alexander-Spanier and $\CF$ech cohomology groups 
are isomorphic. 

\section{Proof of Theorem \ref{TH2}}

Let $T$ be a linear enriched functor.
We define a bifunctor $\FF:\NGc_0^{op}\times \NGcc_0 \to \NGc_0$ as follows.
For $X \in \NGc_0$ and $Y\in \NGcc_0$, we put
\[  
\displaystyle{\FF(X,Y)=\inlim_\lambda \map_0(X_\lambda,
~\holim_{\mu_i} T(Y_{\mu_i}^{\CF}))},
\]
where $\lambda$ is an open covering of $X$ and 
$\{\mu_i \}_{i\geq 0}$ is a set of finite open covers of $Y$ 
such that
$\mu_0 = \{Y\}$, $\mu_i$ is a refinement of $\mu_{i-1}$, 
and
$Y$ is the inverse limit $\pnlim_i Y^{\CF}_{\mu_i}$.

Given based maps $f:X\to X'$ and $g:Y\to Y'$, 
we define a map
\[
\FF(f,g) \in 
\map_0(\FF(X',Y),\FF(X,Y'))
\]
as follows. 
Let $\nu$ and $\gamma$ be open covering
of $X'$ and $Y'$ respectively, and
let $f^{\#}\nu =\{f^{-1}(U)~|~U\in \nu \}$ and $g^{\#}\gamma=\{g^{-1}(V)~|~V\in \gamma \}$.
Then $f^{\#}\nu$ and $g^{\#}\gamma$ are 
open coverings of $X$ and $Y$ respectively.
By the definition of the nerve,
there are natural maps 
$
f_\nu :X_{f^{\#}\nu}\to X^{'}_\nu 
$
and
$
g_\gamma :Y^{\CF}_{g^{\#}\gamma}\to (Y')^{\CF}_\gamma .
$
Hence we have the map 
\[
T(g_\gamma)^{f_\nu}:T(Y^{\CF}_{g^{\#}\gamma})^{X'_\nu}\to T((Y')^{\CF}_{\gamma})^{X_{f^{\#}\nu}}
\]
induced by $f_\nu$
and
$g_\gamma$.
Thus we can define
\[
\FF(f,g)=\inlim_\nu \holim_\gamma T(g_\gamma)^{f_\nu}
:\FF(X',Y)\to \FF(X,Y').
\]
~\\
\textbf{Theorem \ref{TH2}.}
The functor \v{F} is a bilinear enriched functor.

First we prove the bilinearity of \v{F}.
For $Z$,
we prove that
the sequence 
\[
\FF (X\cup CA,Z)\to \FF (X,Z)\to \FF (A,Z)\]
is a homotopy fibration sequence.
Let $A\to X \to X\cup CA$ be a cofibration sequence.
Let $\lambda$ be an open covering of $X\cup CA$, and
let   %
   $\lambda_X$, $\lambda_{CA}$ and $\lambda_A$ 
be consists of those $U\in \lambda$ such that
$U$ intersects with $X$, $CA$, and $A$, respectively.
We need the following lemma.

\begin{Lemma}\label{heq}
We have a homotopy equivalence
   \[
   (X\cup CA)_\lambda^{\CF} 
   \simeq X_{\lambda_X}^{\CF}\cup C(A_{\lambda_{A}}^{\CF}).
   \]
\end{Lemma}
\begin{proof}
By the definition of the $\CF$ceh nerve, we have 
$   (X\cup CA)_\lambda^{\CF} = X_{\lambda_X}^{\CF}\cup (CA)_{\lambda_{CA}}^{\CF}$.
By the homotopy equivalence 
\[
A_{\lambda_{A}}^{\CF}=A_{\lambda_{A}}^{\CF} \times \{0\} \simeq A_{\lambda_{A}}^{\CF}\times I
\]
where $I$ is the unit interval,
we have
\[
X_{\lambda_X}^{\CF}\cup (CA)_{\lambda_{CA}}^{\CF}
\simeq X_{\lambda_X}^{\CF}\cup A_{\lambda_A}^{\CF}\times I\cup (CA)_{\lambda_{CA}}^{\CF}.
\]
Since $(CA)_{\lambda_{CA}}^{\CF}\simeq *$, we have 
\[
X_{\lambda_X}^{\CF}\cup A_{\lambda_A}^{\CF}\times I\cup (CA)_{\lambda_{CA}}^{\CF}
   \simeq X_{\lambda_X}^{\CF}\cup C(A_{\lambda_{A}}^{\CF}).
\]
Hence we have $(X\cup CA)_\lambda 
   \simeq X_{\lambda_X}^{\CF}\cup C(A_{\lambda_{A}}^{\CF}).$
\end{proof}
By Proposition \ref{D} and Lemma \ref{heq},
the sequence
  \[ A_{\lambda_A}\to X_{\lambda_X} \to (X\cup CA)_\lambda 
  \]
is a homotopy cofibration sequence.
Hence the sequence \[
[(X\cup CA)_\lambda ,Z]\to [X_{\lambda_X},Z]\to [A_{\lambda_A},Z]
\]
   is an exact sequence for any $\lambda$.
Since the nerves of the form $\lambda_X$ (resp. $\lambda_A$) are cofinal in the set of nerves
of $X$ (resp. $A$),
we conclude that
the sequence 
\[
\FF (X\cup CA,Z)\to \FF (X,Z)\to \FF (A,Z)\]
   is a homotopy fibration sequence.

Now we show that the sequence 
  $\FF (Z,A)\to \FF (Z,X)\to \FF (Z,X\cup CA)$
   is a homotopy fibration sequence.
By the linearity of $T$, the sequence
   \[
   T(A^{\CF}_{\lambda_A})\to T(X^{\CF}_{\lambda_X})\to T((X\cup CA)^{\CF}_\lambda)
   \]
   is a homotopy fibration sequence.
   Since the fibre $T(A^{\CF}_{\lambda_{A}})$ is a homeomorphic to the inverse limit
   \[
   \pnlim (*\to T((X\cup CA)^{\CF}_\lambda) 
   \leftarrow T(X^{\CF}_{\lambda_X})),
   \]
   we have
   \[
   \displaystyle \begin{array}{l}
   \pnlim (*\to \holim_{\lambda}T((X\cup CA)^{\CF}_\lambda) 
   \leftarrow \holim_{\lambda_{X}}T(X^{\CF}_{\lambda_X}))
    \\
~~   \simeq \pnlim ~\holim_{\lambda}(*\to T((X\cup CA)^{\CF}_\lambda) 
   \leftarrow T(X^{\CF}_{\lambda_X}))\\
~~   \simeq \holim_{\lambda}\pnlim (*\to T((X\cup CA)^{\CF}_\lambda) 
~~   \leftarrow T(X^{\CF}_{\lambda_X}))\\
~~   \simeq \holim_{\lambda}T(A^{\CF}_{\lambda_A}).\\
   \end{array}
   \]
This implies that
the sequence 
   \[
   \holim_{\lambda_A} T(A_{\lambda_A}^{\CF})\to 
   \holim_{\lambda_X} T(X_{\lambda_X}^{\CF})\to 
   \holim_{\lambda}T((X\cup CA)_\lambda^{\CF})
   \] 
is a homotopy fibration sequence,
hence so is $\FF (Z,A)\to \FF (Z,X)\to \FF (Z,X\cup CA)$.

Next we prove the continuity of $\FF$.
Let $F(X,Y)=\map_0(X,\holim_{\mu_i} T(Y^{\CF}_{\mu_i}))$,
so that we have $\FF(X,Y)=\inlim_\lambda F(X_\lambda, Y)$.
We need the following lemma.

\begin{Lemma}\label{TH3}
The functor F is an enriched bifunctor.
\end{Lemma}

\begin{proof}
  Let $F_1(Y)=\holim_{\mu_i} T(Y_{\mu_i}^{\CF})$ and
  $F_2(X,Z)=\map_0(X,Z)$,
  so that we have $F(X,Y) = F_2(X,F_1(Y))$.
Clearly $F_2$ is continuous.

Let $G_1$ be the functor maps $Y$ to $\holim_{\mu_i} Y_{\mu_i}^{\CF}$.
Since $T$ is enriched,
$F_1$ is continuous if and only if 
so is $G_1$.
It suffices to show that
the map $G'_1 \colon \map_0(Y,Y') \times \holim_{\mu_i} Y_{\mu_i}^{\CF} \to \holim_{\lambda_j} (Y')_{\lambda_j}^{\CF}$, adjoint to $G_1$,
  is continuous for any $Y$ and $Y'$.
  Given an open covering $\lambda$ of $Y^{'}$,
  let $p^n_{\lambda}$ be the natural map
  $\holim _\lambda (Y')^{\CF}_\lambda \to \map_0(\Delta^n,(Y')^{\CF}_{{\lambda}})$.
 Then $G'_1$ is continuous if so is the composite
  \[
  p^n_{\lambda}\circ G'_1 \colon  \map_0(Y,Y') \times \holim_{\mu_i} Y_{\mu_i}^{\CF} \to
  \map_0(\Delta^n,(Y')^{\CF}_{\lambda})
  \]
  for every $\lambda \in \mbox{Cov}(Y')$ and every $n$.

  Let $(g,\alpha) \in \map_0(Y,Y') \times \holim_{\mu_i} Y^{\CF}_{\mu_i}$, and
  let $W_{K,U} \subset \map_0(\Delta^n,(Y')^{\CF}_{{\lambda}})$ be an open neighborhood of
  $p^n_{\lambda}(G'_1(g,\alpha))
$,
  where $K$ is a compact set of $\Delta^n$ and $U$ is an
open set of $(Y')^{\CF}_\lambda$.

  Let us choose simplices $\sigma$ of $Y_{g^{\sharp}\lambda}^{\CF}$
  with vertices $g^{-1}(U(\sigma,k)),$ 
where $U(\sigma,k)\in \lambda$ for
  $0 \leq k \leq \dim\sigma$. 
  Let
  \[
  \textstyle
  O(\sigma) = \bigcap_{0 \leq k \leq \dim\sigma} U(\sigma,k)
  \subset Y'.
  \]
Let us choose a point $y_\sigma \in \bigcap_{0 \leq k \leq \dim\sigma} g^{-1}(U(\sigma,k))$,
then $g(y_\sigma )\in O(\sigma)$.
  Let $W_1$ be the intersection of all $W_{y_\sigma,O(\sigma)}$.

There is an integer $l$ such that 
\[
\mu_l>\overline{\mu}_l>g^{\#}\lambda
\] 
where $\overline{\mu_l}$ is a closed covering $\{ \overline{V}|V\in \mu_l \}$
of $Y$.
Thus for any $U\in \mu_l$, there is an open set
$V_U\in g^{\#}\lambda$ such that $\overline{U}\subset g^{-1}(V_U)$.
Since $Y$ is a compact set, $\overline{U}$ is compact.
Let $W_2$ be the intersection of $W_{\overline{U},V_U}$,
and let $W=W_1\cap W_2$.

Since $\mu_l>g^{\#}\lambda$, we have 
\[
p^n_\lambda (G'_1(g,\alpha))=(g_\lambda)_* (\pi^{\mu_l}_{
g^{\#}\lambda})_*p^n_{\mu_l}\alpha.
\]
where $(g_\lambda)_*$ and $(\pi^{\mu_l}_{
g^{\#}\lambda})_*$ are induced by $g_\lambda :Y_{g^{\#}\lambda}^{\CF}\to (Y')_{
\lambda}^{\CF}$ and
$\pi^{\mu_l}_{
g^{\#}\lambda}: Y_{\mu_l}^{\CF}\to Y_{
g^{\#}\lambda}^{\CF}$, respectively.
Let
\[
W'=(p^n_{\mu_l})^{-1}(W_{K,(\pi^{\mu_l}_{
g^{\#}\lambda})^{-1} (g_\lambda)^{-1}(U) }).
\]
Then $W \times W'$ is a neighborhood of $(g,\alpha)$ in
  $\map_0(Y,Y') \times \holim_{\mu_i} Y_{\mu_i}$.
  To see that $p_{\lambda}\circ G'_1$ is continuous at $(g,\alpha)$,
  we need only show that $W \times W'$ is contained in
  $(p_{\lambda} \circ G'_1)^{-1}(U)$.
  Suppose $(h,\beta)$ belongs to $W \times W'$.
  Since $W$ is contained in $W_1$, we have
  \[
  \textstyle y_\sigma \in h^{-1}(O(\sigma)) \subset
  \bigcap_{0 \leq k \leq \dim\sigma} h^{-1}(U(\sigma,k)).
  \]
  This means that the vertices $h^{-1}(U(\sigma,k)) \in h^{\sharp}\lambda$,
  $0 \leq k \leq \dim\sigma$,
  determine simplices $\sigma'$ of $Y_{h^{\sharp}\lambda}$
 each corresponding to each
  $\sigma \subset Y_{g^{\sharp}\lambda}$.
Thus we have an isomorphism
\[
s:Y_{h^{\sharp}\lambda}^{\CF}\to Y_{g^{\sharp}\lambda}^{\CF},
\] 
\[
h^{-1}(U(\sigma,k))\mapsto g^{-1}(U(\sigma,k)).
\]
Moreover since $W$ is contained in $W_2$,
we have $\overline{\mu_l}>h^{\#}\lambda$.

Since the commutative diagram
\[  
\xymatrix{
Y_{\mu_l}^{\CF}\ar[r]\ar[rd]&Y_{g^{\#}\lambda}^{\CF} \ar[r]^{g_\lambda}&(Y')^{\CF}_\lambda\\
&Y_{h^{\#}\lambda}^{\CF}\ar[ru]_{h_\lambda}\ar[u]^s& 
}\]
is commutative,
we have the equation
\[
p^n_{\lambda}\circ G'_1(h,\beta)(K)
=%
  h_{\lambda}\pi^{\mu_l}_{h^{\#}\lambda}(\beta)(K)
=
 g_{\lambda}\pi^{\mu_l}_{g^{\#}\lambda}(\beta)(K)
\]
Since $g_{\lambda}\pi^{\mu_l}_{g^{\#}\lambda}(\beta)(K)$ 
 is continued in $U$, so is
$p^n_{\lambda}\circ G'_1(h,\beta)(K)$.

  Thus $p^n_{\lambda}\circ G'_1$ is continuous for all
  $\lambda \in \mbox{Cov}(Y')$, and hence so is
  $G'_1 \colon \map_0(Y,Y') \times \holim_{\mu_i} Y_{\mu_i}^{\CF} \to \holim_{\lambda_j} (Y')_{\lambda_j}^{\CF}$.
\end{proof}

We are now ready to prove Theorem \ref{TH2} of $\FF$.
  For given pointed spaces $X$, $Y$ and a covering $\mu$ of $X$,
  let $i_{\mu}$ denote the natural map
  $F(X_{\mu},Y) \to \inlim_{\mu}F(X_{\mu},Y)$.
  To prove the theorem, it suffices to show that the map
  \[
\begin{array}{ccl}
  \textstyle
  \FF' \circ~(1\times i_\lambda) \colon \map_0(X,X') \times F(X'_{\lambda},Y)
&\to& \map_0(X,X') \times \inlim_{\lambda}F(X'_{\lambda},Y)\\
 & \to& \inlim_{\mu}F(X_{\mu},Y)
\end{array} 
 \]
  which maps $(f,\alpha)$ to
  $i_{f^{\sharp}\lambda}(F(f_\lambda ,1_Y)(\alpha))$,
  is continuous for every covering $\lambda$ of $X$.

Let $\textstyle
 R \colon \map_0(X,X')
  \to \inlim_{\mu}\map_0(X_\mu,X'_\lambda) 
$ be the map which assigns to $f:X\to X'$
the image of $\map_0(X,X')$, $f_\lambda \in \map_0(X_{f^\sharp\lambda}, X'_\lambda)$ in
$\inlim_{\mu}\map_0(X_\mu,X'_\lambda)$, 
and let $Q$ be the map \[
\inlim_{\mu}\map_0(X_\mu,X'_\lambda) \times F(X'_{\lambda},Y)
  \to \inlim_\mu F(X_{\mu},Y),\] \[[f ,\alpha ] \mapsto i_{f^{\sharp}\lambda}f_\lambda \circ \alpha  =i_{f^{\sharp}\lambda}(F(f_\lambda ,1_Y)(\alpha) ).
\] 
Since we have $\FF'=Q\circ (R\times 1)$, we need only 
show the continuity of $Q$ and $R$.
Since $Q$ is a composite of elements of Im$R$ and
$F(X'_{\lambda},Y)$, $Q$ is continuous.

To see that $R$ is continuous, let
$W_{K^f,U}$ be a neighborhood of $f_{\lambda}$ in
  $\map_0(X_{f^{\sharp}\lambda},X'_\lambda)$, where
  $K^f$ is a compact subset of $X_{f^{\sharp}\lambda}$ and
  $U$ is an open subset of $X'_\lambda$.
Since $K^f$ is compact, there is a finite subcomplex $S^f$ of $X_{f^{\sharp}\lambda}$ such that $K^f\subset S^f$.
Let $\tau^f_i,~0\le i\le m,$ be simplexes of $S^f$. 
By taking a suitable subdivision of $X_{f^{\sharp}\lambda}$, 
we may assume that
there is a simplicial neighborhood $N_{\tau^f_i}$ of each
$\tau^f_i$, $1\le i \le m$, such that $K^f\subset S^f\subset \cup_i N_{\tau^f_i}\subset f_\lambda^{-1}(U)$.

Let $\{ x_k^i \}$ be the set of vertices of $\tau^f_i$ and let
$W$ be the intersection of all $W_{\{ x^i_k \},U_{(\tau^f_i)'}}$
where $U_{{\tau^f_i}'}$ is an open set of $X'_\lambda$ containing the set $\{ f(x^i_k) \}$.
Then $W$ is a neighborhood of $f$.
We need only show that $R(W)\subset i_{f\#\lambda}(W_{{K^f},U})$. 
Suppose that $g$ belongs to $W$.
Since $\{ x^i_{k}\}$ is contained in $g^{-1}(U_{({\tau^f_i})'})$ for any $i$, a simplex $\tau_i^{g}$
spanned by the vertices is contained in $X_{g^\sharp \lambda}$.
Let $S^g$ be the finite subcomplex of $X_{g^\sharp \lambda}$ consists of simplexes $\tau^g_i$.
By the construction, $S^f$ and $S^g$ are isomorphic.
Moreover there is a compact subset $K^g$ of $X_{g^\sharp \lambda}$ such that 
$K^{g}$ and $K^f$ are homeomorphic.
On the other hand, since $g(\{ x^i_{k}\})\subset U_{({\tau^f_i})'}$, there is a simplex of $X_\lambda'$ having $g_\lambda (\tau_i^{g})$ and $({\tau^f_i})'$ as its faces.
This means that $g_\lambda (\tau_i^{g})\subset f_\lambda (\cup_i N_{\tau^f_i})$.
Thus we have $g_\lambda (K^{g})=\cup_i g_\lambda (\tau_i^{g})\subset f_\lambda (\cup_i N_{\tau^f_i})$.

Let $f^\sharp \lambda \cap g^\sharp \lambda$ be an open covering 
\[
\{ f^{-1}(U)\cap g^{-1}(V) ~|~U,V\in \lambda  \}
\] of $X$.
We regard $X_{f^\sharp \lambda}$ and $X_{g^\sharp \lambda}$ as a subcomplex of 
$X_{f^\sharp \lambda \cap g^\sharp \lambda}$.
Since $g_\lambda |X_{f^\sharp \lambda \cap g^\sharp \lambda}$ is contiguous to $f_\lambda |X_{f^\sharp \lambda \cap g^\sharp \lambda}$,
we have a homotopy equivalence 
$g_\lambda |X_{f^\sharp \lambda \cap g^\sharp \lambda} \simeq f_\lambda |X_{f^\sharp \lambda \cap g^\sharp \lambda}$. 
By the homotopy extension property of $g_\lambda|X_{f^\sharp \lambda \cap g^\sharp \lambda}:X_{f^\sharp \lambda \cap g^\sharp \lambda}\to X_\lambda'$ and $f_\lambda :X_{f^\sharp \lambda} \to X_\lambda'$,
$g_\lambda|X_{f^\sharp \lambda \cap g^\sharp \lambda}$ extends to map $G:X_{f^\sharp \lambda}\to X'_\lambda$.

We have the relation $G\sim
\pi^{f^\sharp \lambda \cap g^\sharp \lambda}_{f^\sharp \lambda}G=g_\lambda|X_{f^\sharp \lambda \cap g^\sharp \lambda}=\pi^{f^\sharp \lambda \cap g^\sharp \lambda}_{g^\sharp \lambda}g_\lambda \sim g_\lambda$, where
$\sim$ is the relation of the direct limit.
Moreover by $G(K^f)\subset f_\lambda (\cup_i N_{\tau^f_i}) \subset U$,
we have $[g_\lambda]=[G] \in i_{f\# \lambda}(W_{K^f,U})$.
Hence $R$ is continuous, and so is $\FF'$.
 
\section{Proofs of Theorems \ref{Sp} and \ref{Cor1}}
To prove Theorems \ref{Sp} and \ref{Cor1}, we need several lemmas.
\begin{Lemma}\label{open}
There exists a sequence $\lambda^n_1<\lambda^n_2 <\cdots 
<\lambda^n_m<\cdots$ of open coverings of $S^n$ such that :
\begin{enumerate}
\item For each open covering $\mu$ of $S^n$, 
there is an $m\in \N$ such that $\lambda^n_m$ is a refinement of $\mu$:
\item For any $m$, $S^n_{\lambda_m}$ is homotopy equivalent to $S^n$.
\end{enumerate}
\end{Lemma}

\begin{proof}
We prove by induction on $n$.
For $n=1$, we define an open covering $\lambda^1_m$
of $S^1$ as follows.
For any $i$ with $0\leq i <4m$, we put 
\[
U(i,m)=\{ (\cos \theta ,~\sin \theta )\mid \frac{i-1}{4m}\times 2\pi +
\frac{1}{16m}\times 2\pi <\theta < \frac{i+1}{4m}\times 2\pi +\frac{1}{16m}
\times 2\pi \}.
\] 
Let $\lambda^1_m=\{U(i,m)|~0\leq i <4m \}$. 
Then the set $\lambda^1_m$ is an open covering of $S^1$ and 
is a refinement of $\lambda^1_{m-1}$.
Clearly $(S^1)^{\CF }_{\lambda^1_m}$ is homeomorphic to $S^1$, hence $S^1_{\lambda^1_m} $ is homotopy equivalent to $S^1$.
Moreover for any open covering $\mu$ of $S^1$, 
there exists an $m$ such that $\lambda^{1}_{m}$ is a refinement of $\mu$. 
Hence the lemma is true for $n= 1$.
Assume now that the lemma is true for $1\le k\le n-1$. Let
${\lambda'}^n_m$ be the open covering $\lambda^{n-1}_m\times 
\lambda^{1}_m  $
of $S^{n-1}\times S^1$ and
let $\lambda_m^{n}$ be the open covering of 
$S^n$ induced by the natural map 
$p:S^{n-1}\times S^1 \to S^{n-1}
\times S^1/S^{n-1}\vee S^1$. 
Since
$S^{n-1}_{\lambda^{n-1}_m}$ is a homotopy equivalence of $S^{n-1}$,
we have
\[
S^n_{\lambda^n_{m} }\approx (S^{n-1}\times S^1/S^{n-1}\vee S^1)_{\lambda^n_m} 
\approx (S^{n-1}_{\lambda_m^{n-1}}
\times S^1_{\lambda_m})/(S^{n-1}_{\lambda^{n-1}_{m}}\vee S^1_{\lambda_m})
\approx S^n.
\]
Thus the sequence $\lambda_1^n <\lambda_2^n <\cdots <
\lambda_m^n <\cdots $ satisfies the required conditions.
\end{proof}

\begin{Lemma}
$h_n(X, \FF) \cong \pi_n\holim_\mu T(X_\mu^{\CF}) $ for $n\geq 0$.
\end{Lemma}
\begin{proof}
By Lemma \ref{open}, we have an isomorphism
\[\inlim_\lambda [S^n_{\lambda},\holim_\mu T(X_{\mu}^{\CF})]\\
\cong [S^n,\holim_\mu T(X_{\mu}^{\CF})].\]
Thus we have
\[
\displaystyle \begin{array}{ccl}
h_n(X,\FF)&=&\pi_0\FF (S^n,X)\\
&=&\pi_0\cholim_\lambda \map_0( S^n_{\lambda},\holim_\mu T(X_{\mu}^{\CF}))\\
&\cong& \inlim_\lambda [S^0,\map_{0}(S^n_{\lambda},\holim_\mu T(X_{\mu}^{\CF})]\\
&\cong& \inlim_\lambda [S^n_{\lambda},\holim_\mu T(X_{\mu}^{\CF})]\\
&\cong& [S^n,\holim_\mu T(X_{\mu}^{\CF})]\\
&\cong& \pi_n\holim_\mu T(X_{\mu}^{\CF}).
\end{array}
\]
\end{proof}

Now we are ready to prove Theorem 2.
Let $X$ be a compact metric space and let $\Ss=\{ T(S^k)~|~k\geq 0\}$.
Since $X$ is a compact metric space, 
there is a sequence 
$\{\mu_i \}_{i\geq 0}$ of finite open covers of $X$ with $\mu_0=X$
and $\mu_i$ refining $\mu_{i-1}$ such that $X=\pnlim_iX^{\CF}_{\mu_i}$ holds.
Let us denote $X_{\mu_i}^{\CF}$ by $X_i^{\CF}$ 
and $X_{\mu_i}$ by $X_i$
if there is no possibility of
 confusion.
According to [F], there is a short exact sequence
\[
\xymatrix{
0 \ar[r] &\pnlim^1_{i} {H}_{n+1}(X_i^{\CF},\Ss)  \ar[r] &
 H^{st}_n(X,\Ss) \ar[r] 
& \pnlim_i{H}_{n}(X_i^{\CF},\Ss) \ar[r] & 0 \\ 
}
\]
where $H_n(X,\Ss)$ is the homology group of $X$ with coefficients in the spectrum $\Ss$.
(This is a special case of the Milnor exact sequence \cite{milnor}.)
On the other hand, by [BK], we have the following.
\begin{Lemma}\label{bous}$ ($\cite{bous}$) $
There is a natural short exact sequence
\[
\xymatrix{
0 \ar[r] &\pnlim^1_{i}~ \pi_{n+1}T(X_i^{\CF})  \ar[r] & \pi_n\holim_i T(X_i) \ar[r] 
&\pnlim_{i} ~\pi_{n}T(X_i^{\CF}) \ar[r] & 0. \\ 
}
\]
\end{Lemma}

By Proposition \ref{numerical},
we have a diagram\\
\begin{equation}
\small{\xymatrix{
0 \ar[r] &\pnlim^1_{i}~ H_{n+1}(X_i^{\CF},\Ss) \ar[d]^\cong \ar[r] &
 H^{st}_n(X,\Ss) \ar[r] 
&\pnlim_{i} ~H_{n}(X_i^{\CF},\Ss) \ar[r] \ar[d]^\cong & 0\\ 
0 \ar[r] &\pnlim^1_{i}~ \pi_{n+1}(T(X_i^{\CF})) 
 \ar[r] & \pi_n(\holim_{i}T(X_i^{\CF})) \ar[r]  
&\pnlim_{i} ~\pi_{n}(T(X_i^{\CF})) \ar[r]  & 0. \\ 
}}
\end{equation}
Hence it suffices to construct a natural homomorphism
\[
H^{st}_n(X,\Ss) 
\to \pi_{n}(\holim_iT(X_i^{\CF}))
\]
making the diagram (4.1) commutative.

Since $T$ is continuous, the identity map $X\wedge S^k\to X\wedge S^k$
induces a continuous map $i':X\wedge T(S^k) \to T(X\wedge S^k)$.
Hence we have the composite homomorphism
\[
\displaystyle \begin{array}{ccl}
H^{st}_n(X,\Ss) &=& \pi_{n}\underleftarrow{{\rm holim}}_i(X_{i}^{\CF}\wedge \Ss)\\
&\cong& \inlim_k \pi_{n+k}(\holim_i(X_i^{\CF}\wedge T(S^k)) \\
&\xrightarrow{I}&\inlim_k \pi_{n+k}(\holim_iT(X_i^{\CF}\wedge S^k))\\
&\cong& \pi_{n}(\holim_iT(X_i^{\CF}))\\
\end{array}
\]
in which $I=\inlim_k{i'}^k_*$ is induced by the homomorphisms
\[{i'}^k_*:\pi_{n+k}(\holim_i(X_i^{\CF}\wedge T(S^k))\to \pi_{n+k}(\holim_iT(X_i^{\CF}\wedge S^k)).\]
Clearly the homomorphism 
$H^{st}_n(X,\Ss)\to \pi_{n}(\holim_iT(X_i^{\CF}))
$
makes the diagram (4.1) commutative.
Thus $h_n(X,\FF)$ is isomorphic to the Steenrod homology group coefficients
in the spectrum $\Ss$.

Finally, to prove Theorem 3 it suffices to show that $h^n(X,\CF)$ is isomorphic to the $\CF$ech cohomology group 
of $X$.

By Lemma \ref{open},
we have a homotopy commutative diagram
\[
\xymatrix{
\cdots \ar[r]^= & AG(S^n) \ar[r]^= \ar[d]^\simeq & AG(S^n) 
\ar[r]^= \ar[d]^\simeq&\cdots \\ 
\cdots \ar[r] & AG(S^n_{\lambda^n_{m-1}}) \ar[r]^\simeq & AG(S^n_{\lambda^n_{m}})
 \ar[r] &\cdots .\\ 
}
\]
Hence we have $AG(S^n)\simeq {\underleftarrow{\textrm{holim}}}_i AG(S^n_{\lambda_i^n})$.

Thus we have
\[
\displaystyle \begin{array}{ccl}
h^n(X,\check{C})&=&\pi_0\check{C}(X,S^n)\\
&=&\pi_0\cholim_\lambda \map_0(X_{\lambda},\holim_\mu AG((S^n)^{\CF}_{\mu}))\\
&\cong& [S^0,\cholim_\lambda \map_0(X_{\lambda},AG(S^n)]\\
&\cong& \inlim_\lambda [S^0,\map_{0}(X_{\lambda},AG(S^n)]\\
&\cong& \inlim_\lambda [S^0\wedge X_{\lambda},AG(S^n)]\\
&\cong& \inlim_{\lambda} [X_{\lambda},AG(S^n)].
\end{array}
\]
Hence $h^n(X,\CF )$ is isomorphic to the \v{C}ech cohomology group.

\providecommand{\bysame}{\leavevmode\hbox to3em{\hrulefill}\thinspace}
\providecommand{\MR}{\relax\ifhmode\unskip\space\fi MR }
\providecommand{\MRhref}[2]{%
  \href{http://www.ams.org/mathscinet-getitem?mr=#1}{#2}
}
\providecommand{\href}[2]{#2}

\end{document}